\documentclass[12pt]{amsart}
\usepackage[T1]{fontenc}
\usepackage[utf8]{inputenc}
\usepackage[english]{babel}
\usepackage{extsizes}
\usepackage{amssymb,amsfonts,amsthm,amsmath}
\usepackage{graphicx}
\usepackage{color}
\usepackage{subcaption}
\usepackage[pagebackref=true]{hyperref}
\usepackage{hyperref}
    \hypersetup{
          unicode   = true,
          colorlinks= true,
    }
\usepackage{cite}
\usepackage{setspace}
\AtBeginDocument{
\addtolength\abovedisplayskip{-0.3\baselineskip}
 \addtolength\belowdisplayskip{-0.3\baselineskip}
 \addtolength\abovedisplayshortskip{-0.3\baselineskip}
 \addtolength\belowdisplayshortskip{-0.3\baselineskip}
}
\usepackage[left=1 in,top=1 in,right=1 in, bottom=1 in]{geometry}
\numberwithin{equation}{section}
\definecolor{darkred}{rgb}{.70,.12,.20}

\definecolor{darkgreen}{rgb}{.20,.52,.14}

\definecolor{byz}{rgb}{.44,.16,.39}

\numberwithin{equation}{section}
\newtheorem{remark}{Remark}
\newtheorem{definition}{Definition}

\newtheorem{assumption}{Assumption}
\newtheorem{theorem}{Theorem}

\newtheorem{corollary }{Corollary}
\newtheorem{Lad-Ur}{Ladyzhenskaya-Uraltceva iterative Lemma}

\newtheorem{problem}{Main Problem}
\usepackage{todonotes}

\title[Peaceman Model for Well-Block For Non-Linear Flows Near Well]{Fundamentals in Peaceman Model for Well-Block radius For Non-Linear Flows Near Well }
\author{
 A. Ibraguimov$^1$, E. Zakirov.$^2$, I. Indrupskiy$^3$, D. Anikeev$^4$ 
}
\date{}
\begin{document}
\maketitle
\begin{center}
{
{$^1$ Department of Mathematics and Statistics, Texas Tech University, Oil and Gas Research Institute of Russian Academy of Sciences}
\\
\small{Lubbock, Texas, USA, Moscow,Russia e-mail: \texttt{akif.ibraguimov@ttu.edu}}
\smallskip
\\
{$^2$ Oil and Gas Research Institute of Russian Academy of Sciences,}
\\
{Moscow, Russia, e-mail: \texttt{ezakirov@ogri.ru}}
\\
{$^3$ Oil and Gas Research Institute of Russian Academy of Sciences,}
\\
{Moscow, Russia, e-mail: \texttt{i-ind@ipng.ru}}
\\
{$^4$ Oil and Gas Research Institute of Russian Academy of Sciences,}
\\
{Moscow, Russia, e-mail: \texttt{anikeev@ogri.ru}}
\\
}
\end{center}
\begin{abstract}
\noindent
 We consider sewing machinery between finite difference and analytical solutions defined at different scale: far away and near source of the perturbation of the flow.
One of the essences of the approach is that coarse problem and boundary value problem in the proxy of the source model two different flows. We are proposing method to glue solution via total fluxes, which is predefined on coarse grid.  It is important to mention that the coarse solution "does not see" boundary.

From industrial point of view our report provide mathematical tool for analytical interpretation of simulated data for fluid flow around a well in a porous medium. It can be considered as a mathematical "shirt" on famous Peaceman well-block radius formula for linear (Darcy) radial flow but can be applied in much more general scenario. 

As an important case, we consider nonlinear Forchheimer flow. In the article we rigorously obtained well-block radius, explicitly  depending on $\beta-$Forchheimer factor and total rate of the flow on the well, and provide generalization of the   Dake Formula and evaluation of the $D-$factor.

\end{abstract}
\tableofcontents
\section{Introduction}
Many industrial simulators of the processes of fluid flows  are based on the numerical solution of the partial differential equations (PDE) models (see for example \cite{eclipse},\cite{bogachev},\cite{CMG},\cite{naz1}). Due to large difference in scales, there is a need for analytical approximation to replace numerical solution near the source (boundary) controlling the processes.

Aim of the article is to revisit \textit{Peaceman well block radius},
which is routinely used by reservoir engineers to link value of the pressure for the grid block of numerical solution to the actual  value of the pressure on the well. In the paper we investigate this issue from mathematical point for linear Darcy flow and non-linear Forchheimer type of the flow.   

We reformulate Peaceman method to make analytical and mathematical aspects of the problem transparent.  
In this paragraph, we will highlight main topics of our interest in this project.
Consider baseline case of the 2-D steady regime of the flow initiated in the fully penetrated reservoir of the thickness $h=1$ by sole vertical well of radius $r_w$, with given production rate $q$. Reservoir pressure on the external boundary will be considered fixed. Let generated grid in the domain to be such that source (final) block $Q_0$ of size $\Delta$ contains the well. Consequently, RHS of the corresponding system of the algebraic equations is homogeneous (null) in all cells but $Q_0$.

Let numerical solution of the PDE to be $\tilde{p}$ defined at each block on the grid takes value $p_0$ at $Q_0$. By construction $\tilde{p}$ forms a finite dimensional matrix and is defined at discreet "points". It does not "see" the well of the small radius $r_w$. In the same time well is generating flows in the domain, and it is vital to use $p_0$ for finding $p_w$, which one can assign as pressure value on the well for further processing.  To do so Peaceman considered material balance (or finite-difference) equation in the five-spot system of the grid blocks: $\{1\},\{2\},\{3\},\{4\}$, and  $\{0\}$ (see Fig. \ref{circular grid}) \cite{zia5}:
 \begin{align}\label{Mat-Bal}
        \frac{k}{\mu } \cdot \frac{(p_1 - 2p_0 + p_3)}
        {(\Delta x)^2}) + \nonumber
    \\ 
         + \frac{k}{\mu}  \cdot \frac{(p_2 - 2p_0 + p_4)}{(\Delta y)^2} = \frac{q}{\Delta x\cdot\Delta y}, 
    \end{align}
    where $p_1=p(-\Delta x,0), \ p_3=p(\Delta x,0) \, \ p_2=p(0,\Delta y), \  p_4=p(0,-\Delta y), \ \text{and } p_0= p(0,0). $

Assuming symmetry constraint $p_1 = p_2 = p_3 = p_4$, and $\Delta_x=\Delta_y=\Delta$, Peaceman then reduced \eqref{Mat-Bal} to the equation  
\begin{equation}\label{Mat-Bal-0}
   p_1-p_0=q\frac{\alpha}{4}, \ \alpha=\frac{\mu}{k}. 
\end{equation}
To relate value $p_0$ to the value $p_w$ on the well Peaceman uses Darcy-Dupuit formula, which relates pressure drop $p_{\theta}-p_r$ between value of the pressure $p_\theta=p(\theta),$ on the imaginary well of radius $\theta$ and current pressure $p_r=p(r)$ at radius $\theta <r\leq \Delta$, to the total rate $q:$
\begin{equation}\label{dup-0}
 {-}\alpha^{-1}\frac{\partial p(r)}{\partial r}= v_r\leftrightarrow  p_r=p_\theta+q\alpha \cdot\frac{1}{2\pi}\ln\frac{r}{\theta}.   
\end{equation}

Main question is as follows: \textit{For given $\Delta$, does auxiliary $R_0$ exists such that $p_{R_0}-p_{\Delta}=p_0-p_1$ for any $q$?
}

In the above  $p_{R_0},\ p_\theta,  p_\Delta$ are obtained by Dupuit-Darcy equation of flow for $r=R_0$, $r=\theta$ and $r=\Delta$ respectively. 

Peaceman obtained the remarkable answer \cite{zia5}: \textit{Such $R_0$ exists and is defined via equation $R_0=\Delta \cdot e^{-\pi/2}.$} And this $R_0$ is called the Peaceman well block radius. 

Note that Peaceman's $R_0$ does not depends on $q$ and $\theta.$ 

\begin{figure}
        \centering
        \includegraphics[scale=0.4]{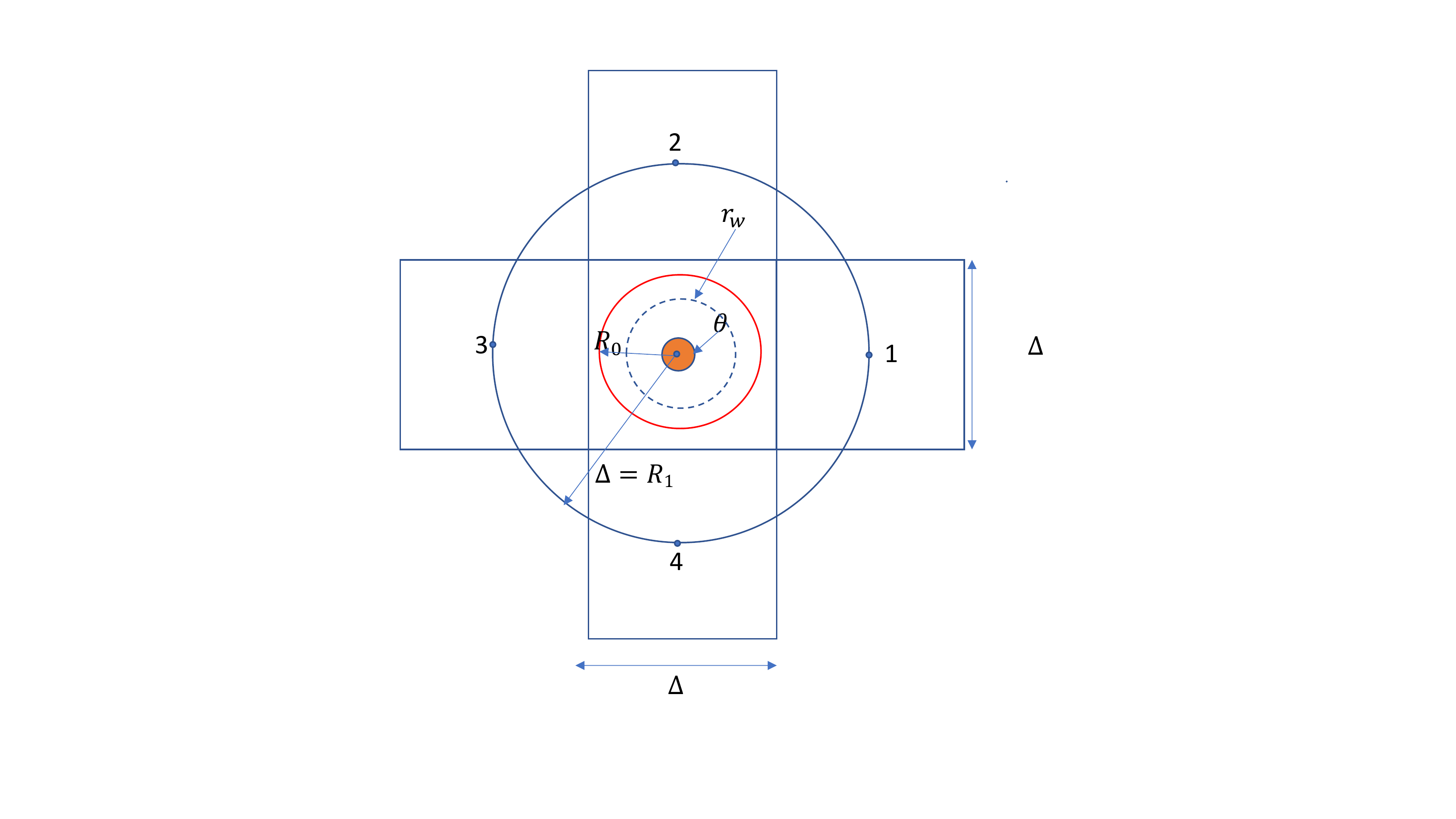}
        \caption{Five-spot grid of size $\Delta$ with the well at center, and auxiliary $U(0,R_w,R_1)$ and $U(0,R_w,R_0)$}
        \label{circular grid}
\end{figure}
    
Then pressure on the well can be reasonably approximated from any numerical solution (value $p_0$) as shown in Sec. \ref{Peace-Rev}.

This issue is discussed in detail in Sec. \ref{sweing-for}, in which we also introduce Peaceman well-posedness in order to make arguments more rigour, which can be used in different scenarios. 

Mathematically material balance equation \eqref{Mat-Bal-0} and  Dupuit-Darcy equation \eqref{dup-0} of flow  are split. In the article we further consider non-linear flows in porous media under hypothesis that
\textit{block-to-block (global) material balance can be assumed linear, and  Forchheimer non-linearity is included only in the equation of flow \eqref{dup-0} and has a form: $-\frac{\partial p(r)}{\partial r}=\alpha v_r+\beta v_r^2$ .} Considering the same arguments as for the linear Darcy flow, this results in a dependency of $R_0$ on $q$.  We implemented this algorithm in Sections \ref{Two-Term} and \ref{Dake}, and obtained formulae for $R_0$ which is very close to traditional one used in simulators if $\Delta$ is large, but quite different when $\Delta$ is small.

In our up-coming research projects we will extend this approach to different classes of pre-Darcy and post-Darcy equations, based on analytical formulae(see \cite{ibragim-prod-ind-gas}) for steady state and PSS solutions for generalised non-Darcy equation.  We will also investigate non-linear Forchheimer-type balance equation of the form $4(p_1-p_0)=\alpha q+f(q),$ for non-linear flow. 

Note that domain of the flow $U=\Omega\setminus B(0,r_w)$ does not contain singular point, and method of the domain discretization based on the material balance is such that $\Delta>r_w.$ This does not allow unbounded grid refinement, consequently theoretical comparison between analytical solution and numerical one is limited.
An interesting mathematical insight for Peaceman formula came from comparison of the numerical solution with Green function which we did in Sec.   \ref{Green-sec}. Namely, finite difference scheme applied for Laplace equation defined in whole domain with isolated source point generates numerical solution which is an approximation to the Green function in the "big" domain $\Omega$  with singularity at $0$ .
Then analogous to the Peaceman well posedness, the question for the analytical Green function $G(x,y)$ is stated as follows:
 \textit{
 Does auxiliary $R_0$ exists such that $\left(G(x,y)\Big|_{|X|=\Delta}-G(x,y)\Big|_{|X|=R_0}\right)\asymp 1+O(\Delta)?$ 
}
In Theorem \ref{Peaceman-for-Green} we proved that the answer is $R_0=\Delta \cdot e^{-\pi/2}.$ Which once more conforms Peaceman formula for $\Delta\asymp 0,$ as well.  

To investigate non-linear flows we propose splitting method in which balance equation is the same as for Darcy flow, but non-linearity impact on the relation between $p_0$ and $p_1$ and value of the pressure on the well is considered through Forchheimer law. This results in the dependence of the auxiliary  $R_0$ on $q$ (see Section \ref{Two-Term}). Obtained relation  justifies formula for $D-$factor , which is actively used by reservoir engineers (see Section \ref{Dake})      
We also formulate general algorithm for well pressure interpretation. It is worth mentioning that some arguments for numerical data interpretation are routinely used by reservoir engineers without mathematical justification. We provide an insight behind picture, which makes clear future engineering application.

Motivation of the research and brief review are presented in the two sections: \ref{well-repres} - Well representation in numerical flow simulations, \ref{trans-prob} - Review on the well transmissibility calculations. Section \ref{conclusion} gives the Concluding remarks. For convenience  we listed all notation in the Section \ref{nomen}.

\section{Nomenclature}\label{nomen}
The following notation is used
\begin{align*}
  &B(0,R)=\{|X|\leq R\}\\  \nonumber
  &X=(x,y)\\ \nonumber
  &v_r  \ - \  \text{radial flow (Darcy) velocity}\\ \nonumber
  &p(x,y)  \ - \   \text{pressure function}\\ \nonumber
  &\delta  \ - \  \text{delta function}\\ \nonumber
  &\Delta  \ - \  \text{grid size}\\  \nonumber
  &k  \ - \  \text{permeability of the porous medium (reservoir)}\\  \nonumber
  &h  \ - \  \text{reservoir thickness}\\  \nonumber
  &\mu  \ - \  \text{viscosity of the fluid}\\  \nonumber
  &R_w \ - \ \text{well radius}\\  \nonumber
  &\theta  \ - \ \text{imaginary well radius }\\  \nonumber
  &\tilde{p}\ - \ \text{numerical solution by discreet scheme of interest defined at discreet points } (i,j).\\  \nonumber
  & f(s)\asymp g(s) , \ \text{if} \ \lim_{s\to 0}\frac{f(s)}{g(s)}=constant >0\\ \nonumber
\end{align*}

%\akif{more details here needed.}  
   
%Conditions on the key  characteristics of generalised Einstein model stated    in term of the integral: $H(u)=\int_0^u h(s) ds .$ In the current project we assumed that
 %  \begin{equation}
   %H(s) = %\left(\int_{s}^{\infty} \frac{\beta}{\tau P(\tau)} d\tau \right)^{-\frac{1}{\beta}} \nonumber
  %    \end{equation}
  %   where $P$ is such that
 %$H(s) \ \text{and} \  s^{-\frac{\beta}{1+\beta}}H^{\beta}(s)  \rightarrow 0 \text{ as } s\rightarrow 0 .$ 
%Under this structural condition we proved  Ladizenskaya-Uraltceva type iterative estimate for specific parabolic Energy functional defined on the solution of degenerate Einstein-Brownian equation.
 
%From obtained estimate  follow the following localisation property 
%\begin{theorem}
%Let $u$ be weak solution of the inequality \begin{equation}
%\frac{\partial H(u)}{\partial t} - \Delta u \leq 0
%\end{equation}
%\end{theorem}
%---------------------------------------------

\begin{section}{About well representation in a 3D flow model}\label{well-repres}

The problem of well representation in numerical flow modelling has long history and peculiar difficulties. Unfortunately, even today it is not possible to confirm that all the perplexities have been already overcome. Well modeling is a classical problem of representing objects of different scale within a single computational module with the requirement to achieve necessary accuracy of simulations. For example, in petroleum reservoir simulation grid block sizes in the horizontal plane are typically 2-3-4 orders of magnitude larger than the well diameter. A typical well radius is of the order of 10 cm, and common elementary numerical grid block dimensions are ~ $10^2$x$10^2$x1 m.

Progress in the availability of computing facilities has led to an increase in the size of the grid and, to a lesser extent, to a decrease in the size of individual grid cells. Over the years, the typical size of a grid cell has decreased insignificantly, while the model sizes (number of grid blocks) have increased by 2-3 orders of magnitude

At the same time, pressure gradients are maximal in the near-well zones, requiring their most accurate representation in these areas. The farther the grid cell is from the well, the less reservoir pressure change in it is associated with a particular change in well operation (sink/source intensity).

Wells are almost never explicitly simulated. This means that pressure in a grid block penetrated by a well differs significantly from the bottom-hole pressure in the same well. This pressure difference causes influx/outflux to/from the grid block. So the problem of well model arises. This model relates well flow rate to the pressure difference that forces this rate.
To the best of our knowledge, the problem of transition from the well grid block pressure to the bottom-hole pressure of an oil well was originally comprehended in the former USSR. Firstly, it was investigated in relation to 2D flow problems solution on electrical integrators with RC (resistance-capacitance) grids \cite{zia1}. Later similar studies were performed for 2D flow problems solution with finite-difference methods \cite{zia2,zia3}. Same approach to gas wells simulation was considered in \cite{zia4}.
In much more recent foreign publications, well connection transmissibility factor was introduced into the formula for well flow rate in a grid block penetrated by the well. The latter was multiplied by the difference between the grid block pressure and the well bottom hole pressure. Well connection transmissibility plays an important role in the flow modeling. It is usually designated as WI - abbreviation for “well index”. Or alternatively as CF - abbreviation for "connection transmissibility factor". The WI (CF) $T_w$ is defined as the ratio between the well volumetric flow rate (at reservoir conditions) $q_w$ and the difference between the grid block pressure $p_0$ and the well bottom hole pressure $p_w$ for single-phase flow of a unit-viscosity fluid:
\begin{equation}
T_w=\frac{q_w \mu}{p_0-p_w} 
\label{aq2.1}
\end{equation}
or
\begin{equation}
q_w=\frac{T_w}{\mu}(p_0-p_w)
\label{aq2.2}
\end{equation}
Formulas \eqref{aq2.1}-\eqref{aq2.2} express the essence of well modeling in numerical flow simulations. A well is usually not simulated explicitly with approximation of its trajectory by a grid mesh. Instead, it only acts as a point-wise source/sink term in flow (mass balance) equations with the flow rate calculated according to the formula \eqref{aq2.2}.
The form of equation \eqref{aq2.2} is very close to the general expression for the well productivity index PI:
\begin{equation}
q_w=\frac{PI}{\mu}(<p>-p_w)
\label{aq2.3}
\end{equation}
where $Q_w$ is the total well flow rate, and $<p>$  is the average reservoir pressure. Relative affinity of the two equations \eqref{aq2.2} and \eqref{aq2.3} explains why some methods developed for calculation  of PI could also be used for calculation of $T_w$.
In different studies, different parameters can be understood by PI, but in any case this parameter relates the inflow at the well and the reservoir-well pressure drop.

In general, the inflow to the well is approximated in different ways depending on the calculation grid. In this paper, we consider only the case of single-phase incompressible flow for a well penetrating a single grid block of a uniform Cartesian grid with finite-difference approximation. In this case, $q_w$ is defined as the total flow rate through the cylindrical borehole surface area with radius $r_w$. For a more general case, equation \eqref{aq2.2} becomes
\begin{equation}
q_{wi}=\sum_i{T_{wi}}M_i(H_i-H_{wi})
\label{aq2.4}
\end{equation}
Here sum is over grid blocks $i$ penetrated by the well. $T_{wi}$ is the CF, or transmissibility, of the grid block $i$ including contributions of the block and well geometries, and of the grid block flow properties (permeability). In many cases $T_{wi}$ is constant, however, it may vary over time in advanced well models. $M_i$ is the mobility in the grid block $i$ including contributions of the fluid properties and fluid-rock interaction. In the incompressible single-phase case:
\begin{equation}
M_i=\frac{1}{\mu},
\label{aq2.5}
\end{equation}
where $\mu$ is the fluid viscosity. In the more complex case, $M_i$ also incorporates characteristics of the two-phase or three-phase flow and depends on the grid block pressure, fluid saturations (volume fractions) and component concentrations.
$H_i-H_{wi}$ is the difference of total fluid potentials (including pressure and gravity) between the grid block $i$ and the well connection in it. In the general case, $H_{wi}$ is related to the well bottomhole pressure $p_w$ through a hydrostatic correction and friction losses, and it depends on the trajectory of the well.
\end{section}

\begin{section}{Brief review of the connection transmissibility computation problem}\label{trans-prob}
In the simplest case, for a vertical well in a homogeneous reservoir, an analytical formula for calculating $T_w$ was presented in the D.Peaceman’s paper \cite{zia5}. The Peaceman's approach is described in detail below. It deals with calculation of so-called equivalent radius $R_0$, so that the well-block pressure $p_0$ can be interpreted as the pressure at $R_0$ in the analytical solution for axisymmetric flow around the well. In the pre-Peaceman period, the approach of paper \cite{zia6} was widely used, with areally averaged pressure used instead of pressure at the equivalent radius. Inaccuracy of this approach was clearly shown in \cite{zia5}. Later, flow simulation software (flow simulators) included only Peaceman-type well models.

For the sake of justice, it was noted above that the problem of individual well consideration in flow simulations was first solved for vertical wells in the USSR, long before the appearance of similar papers in the United States. However, only the papers of D.Peaceman \cite{zia5,zia7} are generally referred to in this connection.

A number of authors have addressed the problem of well CF calculation. Paper \cite{zia7} generalizes results of \cite{zia5} for the case of a single fully penetrating vertical well located in a Cartesian grid with orthotropic-anisotropic medium. Modifications taking into account alternative well orientations were presented in papers \cite{zia8,zia9}. Although those approaches based on spatial variables scaling are still used in existing commercial simulators, in some cases they can lead to significant errors \cite{zia10}.
However, classical Peaceman model does not cope with complex cases, for example, off-centered wells located not in the center of a grid block \cite{zia11,zia12}, or generalized-form grid geometries, for example, deformed grid blocks or local grid refinement near the well \cite{zia11,zia13}.

Approaches combining semi-analytical solutions and finite difference calculations for regions exceeding a single grid block are much more accurate. In \cite{zia14}, productivity index of a horizontal well parallel to a coordinate axis, as previously derived in \cite{zia15}, was used as a control solution. The authors of \cite{zia14} combined this solution with analytical expression for pressure distribution in a cross-cut problem on a Cartesian grid with large aspect ratio (width to height) of grid cells. This made possible calculation of well productivity index by cross-sections of the three-dimensional problem.

There is some gap in representation of the solutions obtained within the classical Peaceman model when further applied to more complex cases. 

Based on the comparison of analytical and numerical results, accurate values of well connection factor in homogeneous reservoirs were determined in \cite{zia16}. For inhomogeneous case, reference solution was obtained by small-scale modeling on Voronoi grids.
In \cite{zia18}, a single inclined well was considered in a laterally infinite layered system. Semianalytical solution derived with the reflection method and analytical solution of the thin-layer theory was combined to the finite-difference pressure solution, with lateral boundary conditions for pressure determined from the analytical solution.
In \cite{zia10}, Laplace equation was solved in elliptic coordinates using geometric transformation taking into account well inclination in an infinite reservoir. This result was used together with numerical solution to determine appropriate well connection transmissibility.
Single-layer potential was applied in \cite{zia18} to calculate stationary pressure distribution in the vicinity of a well. Additionally, grid block transmissibilities were also matched to take into account specifics of radial flow.
Later in paper \cite{zia19} it was proposed to subdivide well inflow problem into two auxiliary subproblems, so that superposition of both subsolutions gives a solution of the original problem. The first subproblem represents a singular flow caused by well operation. And the second one considers regular flow not depending on the well. More precisely, the first problem consists in pressure determination in the case of given inflow rate in an infinite reservoir. The second problem corresponds to well operation with zero flow rate in a limited reservoir, and boundary conditions at the outer boundary are adjusted to the first problem solution at this boundary.
The main distinctive feature of this method lies in transformation of Cartesian coordinates into logarithmic-polar (radial) coordinates. Then it becomes possible to avoid singularity of the solution in the vicinity of a well. However, the second auxiliary problem becomes somewhat worse than on the original grid, since the approximation error increases. Approximation in the new coordinate system is carried out with standard methods of two-point or multi-point approximation, adjusted for logarithmic polarity of the coordinate system. To model a well, a Peaceman-type model was employed in \cite{zia19}. The value of the equivalent radius for the well block could be arbitrary. The only requirement was that the corresponding circle lied inside the grid block and was larger than the circle of the wellbore.

Exact solutions with grids taking into account well trajectories and near-well heterogeneity have also been discussed in the literature (confer finite element-based grids \cite{zia20} and locally elliptical hybrid grids \cite{zia21}).
Regarding wells of non-conventional types, i.e. with arbitrary trajectory or multilateral wells, it should be noted that modern technologies are successful in their construction. Therefore, they become more and more common. This type of wells can penetrate grid blocks in any direction. In addition, this type of well models often use irregular grids to numerically represent various geological features. These grids cover the range from structured curved meshes, multi level grids to completely unstructured grids. Combined application of wells with complex trajectory and progressive grids leads to significant complications in achieving acceptable accuracy for predicting future well behavior.

Effective approach to modeling productivity of unconventional well types during reservoir depletion consists in the use of semi-analytical approaches. Early studies considered single horizontal wells (of infinite conductivity), parallel to one of the sides of a rectangular reservoir. Solution methods employed sequential integral transformation \cite{zia22,zia23} and point Green functions \cite{zia24,zia25,zia26}, leading to solutions in the form of series.
More complex well geometries were considered later in \cite{zia27,zia28,zia29} using numerical integration of differential equations. Well hydraulics (i.e., limited well conductivity) was combined in \cite{zia30} with flow in the reservoir. All the mentioned semi-analytical methods have advantages of limited input datasets and high computational efficiency compared to finite difference modeling. This makes these methods well suited for preliminary assessment of primary production during reservoir depletion.

Early semi-analytical methods, however, were limited to homogeneous systems or almost strictly layered systems \cite{zia17,zia30}. This was a significant limitation, since productivity of unconventional well types can be significantly affected by small-scale heterogeneity in the near-well area. Small-scale heterogeneity could be included in a detailed computational model. But the resulting model could become very difficult to be built and require significant computational time to run.

In this paper we recall the original Peaceman problem to provide its accurate mathematical interpretation and treatment. The corresponding "sewing machinery" is general and can be extended to more complex cases. As an example, we consider the practically-important case of nonlinear Forchheimer flow around the well. Accurate results are obtained for the equivalent radius and CF (or inflow formula) and compared to the Dake's formula generally used in reservoir simulators.

To some extent, this study is based on similar principles with the mathematical treatment of PI calculations presented in \cite{ibragim-Prod-Ind} for linear Darcy flow case in the detail for all three regimes: pseudo-steady, boundary dominated, and steady state.  In case of non-linear Forchheimer flow, the corresponding framework makes sense only for two regimes: pseudo-steady state and steady state, and is presented in \cite{ibragim-prod-ind-PSS}. For gas, corresponding steady was presented in \cite{ibragim-prod-ind-gas}. Main issues here were that in spite of essentially time-dependent nature of the solution, the productivity index, as it was rigorously proven in cited paper, is  time-independent regardless of initial data, or tends to a steady state value exponentially. Therefore, PI can very naturally be used in numerical simulations of the flow to tune a numerical model to real life (observed) data of the reservoir.
\end{section}

\section{Peaceman Well Block Radius and Fundamentals in Finite Difference Solution}
\subsection{Origin of the Peaceman Formula Revisited}\label{Peace-Rev}
Let's in our own words reformulate main idea of the Peaceman paper \cite{zia5}. The aim is to prepare ground for proposed mathematical framework. All other reports from the Peaceman Well Block radius series of articles \cite{zia5}, \cite{zia7}, and articles which follow, although having high importance, are actually extensions and developments of the Peaceman paradigm which we will present in this section.  
Peaceman considered the problem of numerical simulation of flow towards a well, which is subject to Darcy equation. Total rate of production is fixed. In \cite{zia5} flow is assumed symmetrical w.r.t. the well located in the center block $0$ of a 5-point stencil of a numerical finite difference grid as shown in fig.\ref{Flow-Domain}.

The finite difference approximation is as in \eqref{Mat-Bal}. Due to symmetry it is reduced to equation \eqref{Mat-Bal-0}, or for arbitrary reservoir thickness 
 \begin{equation}\label{p_1-p_0-peac}
    p_1-p_0=\frac{\mu}{4kh}q, \ \text{ here} \ \alpha=\frac{\mu}{kh},
\end{equation}  
Consider the vertical well of radius $R_w$ situated in the center of the block $0$ and causing the flow in radial direction only. Assume that $q$ is the fixed total production rate. Then flow in the domain $U(0,R_1,R_w)=B(0,R_1)\setminus B(0,R_w)$ is uniquely characterised by the pressure distribution
 %\begin{figure}[htp]
   %     \centering
    %    \includegraphics[scale=0.5]{Grid-Rw-R0-R1-Center.pdf}
     %   \caption{Five spot Grid of size $\Delta$ and well at a center, and auxiliary $U(0,R_w,R_1)$ and $U(0,R_w,R_0)$.}
      %  \label{Grid-Rw-R0-R1-Center}
  %  \end{figure}
 \begin{equation}\label{p_r-dup}
     p(r)=\frac{\mu q}{2 k h\pi}\ln \frac{r}{R_1}+p_1, \ \text{  here } p_1 \ \text{is the value of pressure at} \ r=R_1 .
 \end{equation}
Now let $R_1=\Delta$, so that $p_1$ is the same as in \eqref{p_1-p_0-peac}, and then
\begin{equation}\label{p_w-p_1-dup}
p(R_w)=  \frac{\mu q}{2 k h\pi}\ln \frac{R_w}{\Delta}+p_1.
\end{equation}
Peaceman paradigm in short can be stated as  
\begin{problem}
Find $R_0$ s.t. the value of $p_0$ in \eqref{p_1-p_0-peac} provides 
\begin{equation}\label{p_w-p_0-dup}
p(R_w)=  \frac{\alpha q}{2\pi}\ln \frac{R_w}{R_0}+p_0.
\end{equation}
\end{problem}
\begin{theorem}\label{peac-wel-blok-R} Assume that total rate of the production $q$ and step size of the grid $\Delta$ are given. Assume that the single fully penetrating well is located at a point chosen as the origin and so being the center of the numerical block $[-\frac{\Delta}{2},\frac{\Delta}{2}]^2$. Let $R_w$ is such that 
$\ln\frac{\Delta}{R_w}>\frac{\pi}{2}$.
Then necessary and sufficient condition that guarantees existence of the solution of the system \eqref{p_1-p_0-peac}-\eqref{p_w-p_0-dup} for $R_0$ is  $\ln\frac{\Delta}{R_0}=\frac{\pi}{2}.$
\end{theorem}
\begin{proof}
Proof follows from direct equivalent algebraic operations on the equations \eqref{p_1-p_0-peac},\eqref{p_w-p_0-dup}, and  \eqref{p_w-p_1-dup}.
\end{proof}
\begin{remark}
In this section interpretation of the Peaceman paper \cite{zia5} is made directly without significant modification. But it is already clear that the main idea is to sew numerical and analytical solutions formulated on different scales to provide required information for tuning procedure between the numerical solution and observed well pressure data. In what will follow in this paper, we will bring more mathematical insight in the Peaceman sewing machinery, and provide a general framework suitable for non-Darcy flows. 
\end{remark}
\subsection{Green function and related interpretation of the well-block radius }\label{Green-sec}
Mathematically speaking, Peaceman numerical scheme as presented in \cite{zia5} is a finite difference representation for the Green function approximated on the square  block with the source at center. In this subsection we state the corresponding problem in PDE format and provide its relation to the numerical solution. Then we will state a problem, provided that $G(x,0)$ is a Green function with source at the origin, of finding the value of $R_0$ as a function of $\Delta$ such that oscillation of the $G(x,0)$ in the annular domain $U(0,R_0,\Delta)=U\cap B(0,R_0,\Delta)  $ almost constant  for $\Delta \asymp 0$ and 
$\Delta \text{ such that} \ B(0,\Delta)\in U.$

Consider BVP in the bounded domain $U$ containing source point $0$ (see fig. \ref{Flow-Domain}) :
\begin{align}\label{bvp-Green}
  \frac{kh}{\mu } \cdot \Delta p = q^0\cdot \delta(x) \ \text{in the domain \ $U$}   \\
  p(x)=0 \ \text{on the boundary} \ \partial U \label{bound-cond}
\end{align}

   \vspace{-0.5cm}
    \begin{figure}[htp]
   \centering
        \includegraphics[scale=0.6]{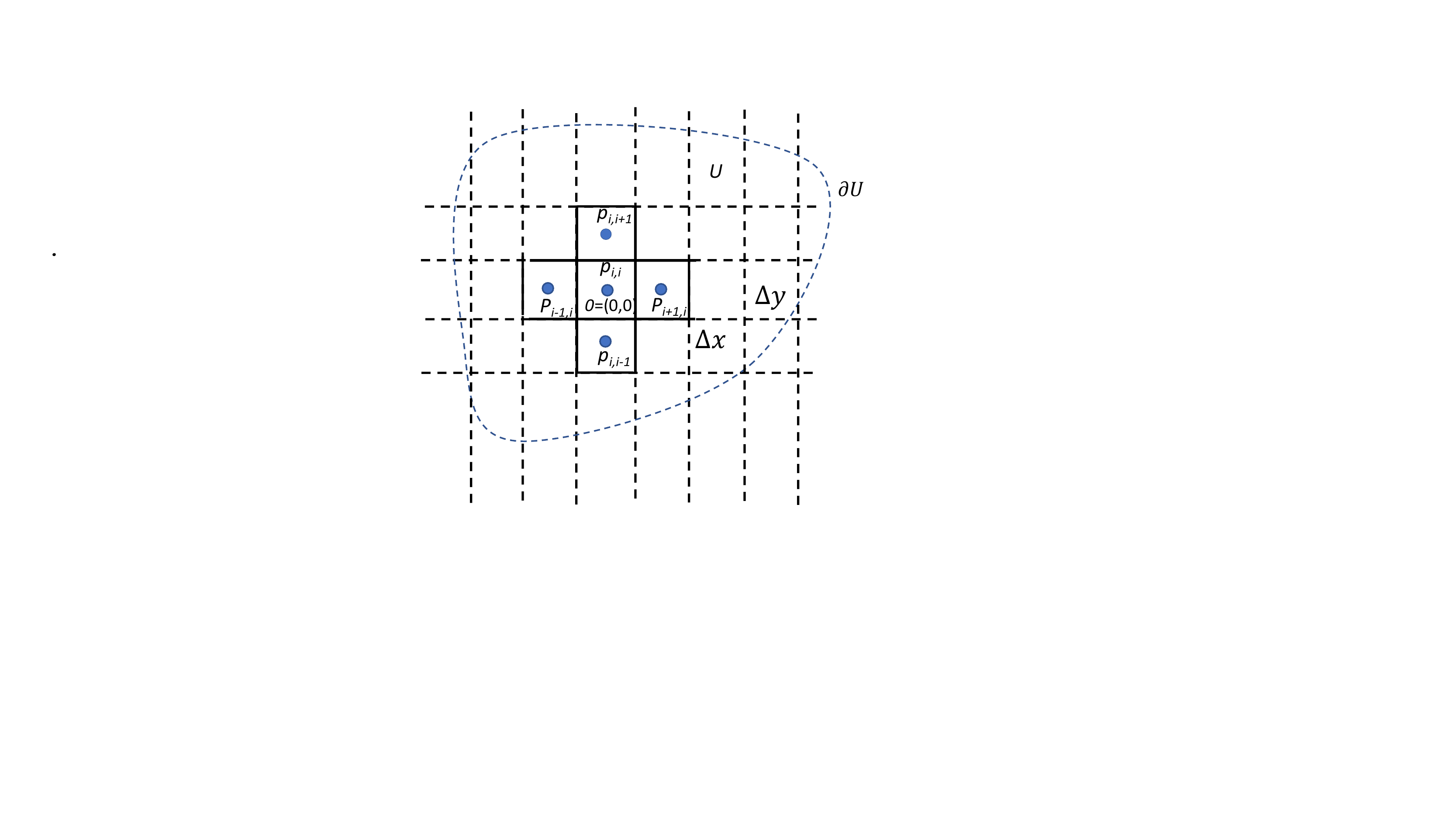}
        \vspace{-3cm}
        \caption{Five-point stencil for the finite difference scheme, and domain of the flow with source at $0$}
        \label{Flow-Domain}
    \end{figure}
    
We divide the whole area of flow into $M\times M$ blocks. For all blocks $0\leq i \leq M,$ $0 \leq j \leq M$ (see fig. \ref{Flow-Domain}), the finite difference equation for \eqref{bvp-Green} will take form:
   \begin{align}\label{a-0}
        \frac{kh}{\mu } \cdot \frac{(p_{i+1, j} - 2p_{i,j} + p_{i-1, j})}
        {(\Delta x)^2}) + \nonumber
    \\ 
         + \frac{kh}{\mu}  \cdot \frac{(p_{i, j+1} - 2p_{i,j} + p_{i, j-1})}{(\Delta y)^2} = \frac{q_{i, j}}{\Delta x\cdot\Delta y} 
    \end{align}
   
or in the form of Peaceman:
    \begin{align}\label{a-01}
        \frac{kh\Delta y}{\mu \Delta x} \cdot (p_{i+1, j} - 2p_{i,j} + p_{i-1, j}) + \nonumber
    \\ 
         + \frac{kh\Delta x}{\mu \Delta y} \cdot (p_{i, j+1} - 2p_{i,j} + p_{i, j-1}) = q_{i, j} 
    \end{align}
    
    In the above equation, $q_{i,j}=0$ if ${i\neq 0}$ or $j\neq 0$ ($q_{i,j}=q\cdot \delta_{i,j}$ - Kronecker symbol). Evidently, sizes of the block in $x$ and $y$ direction are correspondingly $\Delta x$ and $\Delta y$ and converge to $0$ as $M\to \infty$. Let us denote $2M\times2M$ matrix $P_M$
\begin{equation}
P_M=\Big(   p_{i,j} \Big)_{\left((-M\leq i\leq M);(-M\leq j \leq M)\right)}
\end{equation}
 Elements of the matrix represent values of the solution of the discreet Poisson  equation \eqref{bvp-Green} with sink/source localized at center $(0,0)$.
 
 Johansen Taousch from SMU proposed that using classical machinery(see \cite{Evans}) of splitting Green function $G(x,y)$ into the sum of fundamental solution and "corrector" one can prove the following.
 \begin{theorem}
Let $(x,y) \neq (0,0)$ is a fixed point of the domain $U$. This point belongs to one element (block) of the grid $U_M$ which approximates the domain $U.$ Let $p_M(x,y)$ is the solution of the system $2M\times 2M,$ extended to $C^2(U)\cap C^0(\Bar U)$. Then as $M\to \infty$, function $p_M(x,y) \to G(x),$  where \begin{equation}\label{green-function}
G(x)=\frac{1}{2 \pi} \ln \frac{1}{r}+g(x,y), \ r=\sqrt{x^2+y^2}.
\end{equation}
is the Green function.
Function $G$ has the property
\begin{equation}\label{total-rate-green}
\int_{\gamma}q^0\frac{\partial G}{\partial \nu} ds=q^0 ,\  \text{where } \gamma \text{ is a closed contour around the source at} \ 0.
\end{equation}
 \end{theorem}
 We do not prove this Theorem in the article but will address this issue in upcoming publications.

Green function has an interesting property, which directly relate to Peacamn well-block radius.
\begin{theorem}\label{Peaceman-for-Green}
In the axisymmetric case, Green function has the property: 
There exists $R_0$ s.t. if
\begin{equation}\label{G-Delta-G-R0}
   4\cdot\left(G(x)|_{|x|=\Delta}-G(y)|_{|y|=R_0}\right)=1+O(\Delta)
\end{equation}
then 
\begin{equation}
  \ln\frac{\Delta}{R_0}=\frac{\pi}{2} +O(\Delta)
\end{equation}
\end{theorem}

\begin{proof}
First observe that
\begin{equation}
\left(G(x)|_{|x|=\Delta}-G(y)|_{|y|=R_0}\right)=\frac{1}{2\pi}\ln\frac{1}{\Delta}-\frac{1}{2\pi}\ln\frac{1}{R_0}+\left(g(x)\big|_{|x|=\Delta}-g(x)\big|_{|x|=R_0}\right).
\end{equation}
After regrouping from above follows
\begin{equation}
\left(G(x)|_{|x|=\Delta}-G(y)|_{|y|=R_0}\right)=\frac{1}{2\pi}\ln\frac{R_0}{\Delta}+\left(g(x)\big|_{|x|=\Delta}-g(x)\big|_{|x|=R_0}\right).
\end{equation}
Using standard appriory estimate for corrector $g(x)$ as a solution of the Dirichlet problem  one has
\begin{equation}
  \Big|g(x)\big|_{|x|=\Delta}-g(x)\big|_{|x|=R_0}\Big|\leq\max_{x,y\in {U(0,\Delta)}}\big| g(x)-g(y)\big|\leq C\big|x-y\big|\leq 2C\cdot\Delta. 
\end{equation}
Constants $C$ is constant depending on the distance between $0$ and $\partial U$.
\end{proof}

\subsection{Radial Darcy flow in two embedded annulus zones sewed by linear balance equation on the finite difference coarse grid, and its link to capacity   }\label{sweing-for}
In this section we will highlight sewing features behind the well block radius, and indicate its link to capacity from the analytical setup which comes out of the Peaceman method.

Consider auxiliary annular domain $U(0,\theta,\Delta)$, with $\theta<\Delta,$ to be selected by technical arguments. To use Peaceman well block radius as a sewing machinery, assume that flow is radial around the virtual well-disc of radius $\theta$ and is generated by given total flux $q$ on $\Gamma_\theta=\{x: |x|=\theta\}$. Consequently between $\Gamma_\theta$ and exterior boundary $\Gamma_e$ pressure drop will occur specified by given $p_w \text{ and} \ p_e,$ or one of those pressures and the total rate $q$. In the case of fully radial flow also consider the domain $U(0,R_w,R_e)$: $\Gamma_e=\{r=R_e\>>\Delta\}$ and $\Gamma_w=\{r=R_w<\Delta\}$, with actual well radius $R_w$.

%For the annular domain   $p(r)$ and radial velocity $v_r$ have explicit form:
%\begin{equation}\label{p-Darcy-anul}
%p(x,y)=p(r)=(p_e-p_w)\frac{1}{\ln\frac{R_e}{R_w}}\ln%\frac{r}{R_e} +p_e, \ \text{for given} \ p_e , p_w , 
%\end{equation}
%and
% \begin{equation}\label{v-r formula}
% v_r= -\frac{q}{2\pi r}\ \text{for any total rate (over well)} q.
% \end{equation}

Direct integration of the Darcy equation for radial flow \eqref{Darcy-Radial} \begin{equation}\label{Darcy-Radial}
  -\frac{\partial p}{\partial r}=\alpha v_r=\alpha\frac{q}{2\pi r}
\end{equation}
over the domain $U(0,r,R_e)$ provides 
 \begin{equation}\label{closed-form-pressure-drop-Darcy-q}
p(r)-p(R_e)=\frac{\alpha q}{2\pi}\ln\frac{r}{R_e}.
\end{equation}
for any $r>\theta, \ r= \sqrt{x^2+y^2}$ .

%Differences between \eqref{p-Darcy-anul} and above equation due to %two different conditions on the $\Gamma_w$-Dirichlet and Neumann  %and is solved due to explicit relation between factors  %$(p_e-p_w)\frac{1}{\ln\frac{R_e}{R_w}}$ and $\frac{\alpha q}{2\pi}.$ This was established in the Theorem  %\ref{q-in-U(Delta,rw)}, and it closely relate to sewing algorithm. 

Now consider finite difference scheme with axial symmetry. Let
\begin{equation}\label{Omega-0}
\Omega_0=\{(x,y):-L\leq x\leq L \ ; -L \leq y \leq L \},  \Delta_x=\Delta_y=\frac{L}{M}=\Delta \ , \ M \in \mathcal{N}.
\end{equation}
We will impose constraint that for given radius $R_w$ of the "actual" well   $\Delta<\frac{L-R_w}{L}.$ Reasoning behind this constraint will become clear later in this section.

In the rectangular domain $\Omega_0$ consider finite difference BVP
 \begin{align}\label{a-1}
&        \frac{kh}{\mu} \cdot (p_{i+1, j} - 2p_{i,j} + p_{i-1, j}) 
          + \frac{kh}{\mu } \cdot (p_{i, j+1} - 2p_{i,j} + p_{i, j-1}) =  q_{i,j}=q \delta_{i,j} \\
&        p_{-M, j}=p_{M, j}=p_{i, -M}=p_{i, M}=0 \text{ - boundary conditions.} \\
&\text{In above:}\nonumber\\
&\{-M\leq i \leq M \ , \ -M\leq j \leq M\}=I\times J  \, \\
&\delta_{i,j}- \text{index function of the  block $B(0,0)$}:   \delta_{0,0}=1, \delta_{i,j} \text{ equals} \ 0 \ \text{otherwise}.
        \end{align}
        
Denote $2M\times2M$ matrix 
\begin{equation}\label{matrix-P-M}
P_M=\Big(   p_{i,j} \Big)_{\left((-M\leq i\leq M);(-M\leq j \leq M)\right)}
\end{equation}

Evidently matrix $P_M$ exists and is unique. 
Assume $\theta$ be fixed parameter, and 
\begin{equation}\label{L-M-r-w-cond}
\Delta=\frac{L}{M}>\theta.    
\end{equation}
Let $B(0,\theta)$ be the disc model of the well of radius $\theta$.
Following tradition in engineering, we will call the five-point difference equation \eqref{a-1}  for $i=0, j=0$ - the material balance equation.

Matrix $P_M$ is $2D$ approximation of the analytical problem
\begin{align}
  \Delta p=0 \ \text{in} \ \Omega_0 \setminus B (0,\theta); \nonumber \\
  h\frac{k}{\mu}\int_{\Gamma_\theta}\frac{\partial p}{\partial \nu}d s=q ; \label{anal-problem} \\
  p=0 \ \text{on} \ \partial{\Omega_0}.\nonumber
\end{align}
  Let  $\theta << L,$ then it is reasonable to assume that $ p(x,y)\big |_{|x|=\theta}\asymp constant.$    
  
\begin{assumption}\label{symmetry}
Assume symmetry condition in the four blocks surrounding the center block $U(0)=[-\frac{\Delta}{2};\frac{\Delta}{2}]\times[-\frac{\Delta}{2};\frac{\Delta}{2}]$(see Fig. \ref{Flow-Domain} ): 
\begin{equation}
 p_{-1,1}=p_{1,-1}=p_{-1,-1}=p_{1,1}=p_1.   
\end{equation}

\end{assumption}
\begin{definition}\label{material balance}
Denote $p_{0,0}=p_0$, then finite difference scheme in the five spot grid  containing disc(well) $B(0,\theta) $ in the central  block $B_{0,0}$ will take a form  of the symmetric material balance  in the form  :
\begin{equation}\label{p_1-p_0-1}
   p_1-p_0=\frac{\alpha}{4}q.
\end{equation}
\end{definition}

%%%%%%%%%%%%%%%%%%%%

%%%%%%%%%%%%%%%%%%%%%

As it was stated, $p_{i,j}$ is uniquely defined by input parameters $M, q  , L , k , \mu, h.  $ 
Note that $p_1$ as well as $p_0$ depend on the characteristics of the grid size -  $\Delta$, but due to material balance equation their difference is $\Delta$-independent.  

Trace of the solution of the analytical problem \eqref{anal-problem}
on the $|x|=\Delta$ depends on $\Delta$. We need to find $R_0$ as a function of $\Delta$ such that the difference between traces of the analytical solution on $|x|=\Delta$ and $|x|=R_0$ is $\Delta$-independent as well. 
To do so let us introduce two auxiliary BVPs:
\begin{align}
  \Delta u_1=0 \ \text{in} \ U(\Delta,\theta)=B(0,\Delta)\setminus B(0,\theta) ; \nonumber \\
  h\frac{k}{\mu}\int_{\Gamma_\theta}\frac{\partial u}{\partial \nu}d s=q ; \label{B-Delta} \\
  u_1=p_1 \ \text{on} \ \partial{B(0,\Delta)}.\nonumber
\end{align}

\begin{align}
  \Delta u_0=0 \ \text{in} \ U(R_0,\theta)=B(0,R_0)\setminus B(0,\theta) ; \nonumber \\
  h\frac{k}{\mu}\int_{\Gamma_\theta}\frac{\partial u_0}{\partial \nu}d s=q ; \label{B-R-0} \\
  u_0=p_0 \ \text{on} \ \partial{B(0,R_0)}.\nonumber
\end{align}

\begin{definition}
Let imaginary well of the radius $\theta$ ($\theta<R_0<\Delta$) and total rate of the production $q$ are fixed. Denote boundary of the well $\Gamma_\theta=\{|x|=\theta\}.$  
We will say that Peaceman problem in the annular domains $U(R_0,\theta)$ and $U(\Delta,\theta)$ is strictly  well posed if $R_0$ exists s.t. simultaneously two constrains hold: 
\begin{enumerate}
    \item \label{u1=u0}
   solutions of the problems  \eqref{B-Delta} and \eqref{B-R-0} on $\Gamma_\theta$ are constant and 
$$
    u_1\big|_{r=\theta}=u_0\big|_{r=\theta}=p_{\theta},
$$

\item \label{p1-p0-balance}
$p_1$ and $p_0$ are solutions of the balance equation \eqref{p_1-p_0-1}.  
\end{enumerate}

\end{definition}
Another interpretation of the Definition of well-posedness contains sewing property of the material balance equation, which allows to sew the analytical solutions of the problems \ref{B-Delta} and \ref{B-R-0} via values on the well boundary.  
\begin{theorem}\label{Peaceman-basic}
Assume balance equation is as given by \eqref{p_1-p_0-1}. Let imaginary well of the radius $\theta$ to be at the center of a five spot grid. Let number of the blocks $M$, size of the square domain $L$, and imaginary well of radius $\theta$ be such that
\begin{equation}\label{M-delta-theta}
 \Delta=\frac{L}{M}>e^{\frac{\pi}{2}}\cdot \theta.  
\end{equation}
Then for given rate $q$, 
if 
\begin{equation}\label{R-0-delta}
 R_0=\Delta\cdot e^{-\frac{\pi}{2}},   
\end{equation}
then Peaceman problem is strictly well posed.
\end{theorem}
\begin{proof}
Due to radial flow one has \eqref{Darcy-Radial}, and consequently \eqref{closed-form-pressure-drop-Darcy-q}, from which we get the system
\begin{equation}\label{p-delta}
    u_1(\Delta)-u_1(\theta)=p_1-u_1(\theta)=\frac{\alpha q}{2\pi}\ln\frac{\Delta}{\theta}.
\end{equation}
\begin{equation}\label{p-R-0}
    u_0(R_0)-u_0(\theta)=p_0-u_0(\theta)=\frac{\alpha q}{2\pi}\ln\frac{R_0}{\theta}
\end{equation}
Due to balance equation, in order   $R_0$ to exist it is sufficient that   
\begin{equation}
 \frac{\alpha}{4}q-\frac{\alpha q}{2\pi}\left(\ln\frac{\Delta}{\theta}-\ln\frac{R_0}{\theta}\right)=u_1(\theta)-u_0(\theta)  .
\end{equation}
or
\begin{equation}
 \frac{\alpha}{2}q\left(\frac{1}{2}-\frac{1}{\pi}\ln\frac{\Delta}{R_0}\right)=u_1(\theta)-u_0(\theta) . 
\end{equation}
In order the Peaceman problem to be strictly well-posed, it suffices expression in the parentheses to be equal zero:
\begin{equation}
 \left(\frac{1}{2}-\frac{1}{\pi}\ln\frac{\Delta}{R_0}\right)=0 \ \Longleftrightarrow \frac{\pi}{2}=\ln\frac{\Delta}{R_0}.  
\end{equation}
QED
\end{proof}

\begin{remark}

Due to material balance equation there is no difference between $p_1$ and $p_0$ for $\Delta>e^{\frac{\pi}{2}}\cdot r_w$.
But from point of view of the engineering  it is important  to obtain value of the simulated well pressure from pressure values of the simulated numerical solution, more precisely, from value of $\bar{p}_0$  at nearest block containing well  as a function of given $\Delta$.
Algorithm will then consist of two possible scenarios: 
\begin{enumerate}
    \item $\Delta>e^{\frac{\pi}{2}}r_w$. We will find value of the pressure on the well by
    \begin{equation}\label{recont-p-w-1}
     p_w=p_0-\frac{\alpha q}{2\pi}\ln\frac{e^{-\frac{\pi}{2}}\cdot\Delta}{r_w}   
    \end{equation}
    \item $ r_w\leq \Delta\leq e^{\frac{\pi}{2}}r_w$. Then we will let $p_w=p_1-\frac{\alpha q}{2\pi}\ln\frac{\Delta}{r_w}$, or due to balance equation \eqref{p_1-p_0-1} :
    \begin{equation}\label{recont-p-w-2}
     p_w=p_0+q\frac{\alpha}{4} -\frac{\alpha q}{2\pi}\ln\frac{\Delta}{r_w}.
    \end{equation}
    \end{enumerate} 
\end{remark}

\begin{remark}
It is also worth to mention that if $\Delta \to 0$ then $\lim_{\Delta\to 0}p_0(\Delta)$ exist and uniquely determined by $q$. It follows from the following arguments for simple annular domain.
Namely let the radius of imaginary well $\theta=r_w$  .Then   from Theorem \ref{Peaceman-basic} it follows that 
\begin{equation}\label{p-w-Delat}
  p_w(\Delta)=p_0(\Delta)-\frac{\alpha q}{2\pi}\ln\frac{R_0(\Delta)}{r_w} .
\end{equation}

\end{remark}

%%%%%%%%%%%%%%%%%%%%%%%%%%%%

\subsection{Two Terms Radial Case}\label{Two Terms Analitical solution}
Now consider Peaceman problem for nonlinear flow by Forchheimer two-terms law:
 \begin{equation}\label{Radial-Two-Term}
 -\frac{\partial p}{\partial r}=\alpha_1 v_r+\beta_1 v_r|v_r|, \text{ } \alpha_1 = \frac{\mu}{k}=\alpha h
 \end{equation}  
In  \eqref{Radial-Two-Term} if $\beta_1=0$ one can get classical Darcy equation. 
 From \eqref{Radial-Two-Term}
 \begin{equation}\label{Radial-grad-pressure-vs-v-r}
  \frac{\partial p}{\partial r}=\alpha \frac{q}{2\pi r}+\beta \frac{q}{4\pi^2 r^2}, \text{ } \beta = \frac{\beta_1}{h^2}    
 \end{equation}
 Consider flow from  $\partial B(0,R_2)$ to $\partial B(0,R_1)$ in the annular domain $U$(see Fig. \ref{fig:Anul-Domain}): 
 \begin{align}\label{BVP-radial-anal}
 \begin{cases}
 & U=B(0,R_2)\setminus B(0,R_1), \ \text{with fixed total rate } \  q=\int_{S} v(r) ds, ,\\
& \text{Given pressure on one of the boundaries} \ \partial B(0,R_i): \\
& p(r)\big|_{r=R_i}=f_i \ \text{for} \ i=1 \ \text{or}  \ 2 .   
\end{cases}
   \end{align}
Explicit formula for generic solution of \eqref{BVP-radial-anal} for two terms Forchheimer law follows from basic integration:
 \begin{equation}
     p(r)=\frac{\alpha q}{2\pi}\ln r-\beta \frac{q}{4\pi^2}\frac{1}{r}+constant
 \end{equation}
 
Then using boundary conditions in \eqref{BVP-radial-anal} one can get a generic formula for pressure drop between two contours ($\partial B(0,R_i)$) of the boundary of annular domain $U(0,R_1,R_2)=B(0,R_2)\setminus B(0,R_1)$.
 \begin{equation}\label{closed-form-pressure-drop-two-terms-Forch}
p\big|_{r=R_2}-p\big|_{r=R_1}=f_2-f_1=\frac{\alpha q}{2\pi}\ln\frac{R_2}{R_1}+\beta \frac{q}{4\pi^2}\left(\frac{1}{R_1}-
\frac{1}{R_2}\right) 
\end{equation}
\begin{figure}
    \centering
    \includegraphics[scale=0.5]{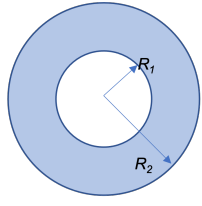}
    \caption{General Annual Domain $U$ }
   
    \label{fig:Anul-Domain}
\end{figure}

\subsection{Criteria for radial non-linear Forchheimer flow in two annular zones with linear balance equation on the finite difference coarse grid  }\label{Two-Term}
In this section we will derive Peaceman-type well pressure formula taking into account non-linearity of the flow near well using Forchheimer correction of Darcy law by two term equation \eqref{Radial-Two-Term}. Although we could use same approach as for the Darcy case, our treatment for non-linear Forchheimer will be different. Namely we will not introduce imaginary radius $\theta$ and corresponding Peaceman well-posedness definition. We do it purposely to highlight main assumption and keep $\theta=R_w$.

Since two-term non-linearity of flow is essential only in near-well high velocity zones, we will adopt conventional engineering hypothesis that on numerical grid material balance can be still considered linear, while near-well correction is due to Forchheimer type of non-linearity. From \eqref{p_1-p_0-peac} and \eqref{closed-form-pressure-drop-two-terms-Forch} with $\beta\neq 0$ the system of 3 equations follows:
\begin{equation}\label{p_1-p_0-D}
   p_1-p_0=\frac{\alpha}{4}q
\end{equation}
\begin{equation}\label{p_1-p_w-F}
   p_1-p_w=\frac{\alpha}{2\pi}q\ln\frac{\Delta}{R_w}+\beta \frac{q^2}{4\pi^2}\left(\frac{1}{R_w}-
\frac{1}{\Delta}\right)   
\end{equation}
\begin{equation}\label{p_0-p_w-F}
   p_0-p_w=\frac{\alpha}{2\pi}q\ln\frac{R_0}{R_w}+\beta \frac{q^2}{4\pi^2}\left(\frac{1}{R_w}-
\frac{1}{R_0}\right)  
\end{equation}
\begin{theorem}\label{glue-radial-Forch}
Assume that $q$ and $p_w$ solve quadratic equation
\begin{equation}\label{q-in-U(Delta,rw)-Forch}
   p_1-p_w=\frac{\alpha}{2\pi}q\ln\frac{\Delta}{R_w}+\beta \frac{q^2}{4\pi^2}\left(\frac{1}{R_w}-
\frac{1}{\Delta}\right)   
\end{equation}
Then if $R_0$ satisfies equation
 \begin{equation}\label{R_0-for-Forch}
    R_0=\Delta\cdot e^{-\delta\frac{\pi}{2}} 
 \end{equation}
 system \eqref{p_1-p_0-D}-\eqref{p_0-p_w-F} has a solution for any mutually related $p_1, \ p_0, \text{and} \ p_w$, 
 if $\delta$ satisfies equation
 \begin{equation}\label{delta-factor-Forch}
  \delta+\beta\frac{q}{\alpha\pi^2}\left(\frac{e^{\delta\frac{\pi}{2}}}{\Delta }-\frac{1}{\Delta} \right) =1  
 \end{equation}
 
\end{theorem}
\begin{proof}
Indeed, let first  substitute \eqref{R_0-for-Forch} into RHS of the \eqref{p_0-p_w-F} to get

\begin{align}
&   p_0-p_w=\left(\frac{\alpha}{2\pi}q\ln\frac{\Delta}{R_w}+\beta \frac{q^2}{4\pi^2}\left(\frac{1}{R_w}-
\frac{1}{\Delta}\right)\right) \label{p_1}\\ &+\left(\frac{\alpha}{2\pi}q(-\delta\frac{\pi}{2})+\beta\frac{q^2}{4\pi^2}\left(-\frac{1}{\Delta\cdot e^{-\delta\frac{\pi}{2}}}+\frac{1}{\Delta}
\right)\right)  \label{q}
\end{align}
   First term in parentheses in RHS of \eqref{p_1} is equal to  $p_1-p_w$, and the second one we will rewrite as
   \begin{equation}\label{p_0-p_w-p_1-p_w-F}
      p_0-p_w=p_1-p_w +q\left(\frac{\alpha}{4}(-\delta)+\beta\frac{q}{4\pi^2}\left(-\frac{1}{\Delta\cdot e^{-\delta\frac{\pi}{2}}}+\frac{1}{\Delta}
\right)\right)   
   \end{equation}
or
\begin{equation}\label{p_0-p_1-F}
      p_0-p_1= -q\frac{\alpha}{4}\left(\delta+\beta\frac{q}{\alpha\pi^2}\left(\frac{1}{\Delta\cdot e^{-\delta\frac{\pi}{2}}}-\frac{1}{\Delta}
\right)\right).  
   \end{equation}
   Finally we have 
 \begin{equation}\label{p_1-p_0-q-F}
      p_1-p_0= q\frac{\alpha}{4}\left(\delta+\beta\frac{q}{\alpha\pi^2}\left(\frac{1}{\Delta\cdot e^{-\delta\frac{\pi}{2}}}-\frac{1}{\Delta}
\right)\right).  
   \end{equation} 
   To get \eqref{p_1-p_0-D} and consequently to finish the proof it will be sufficient to assume that $\delta$ satisfies equation\eqref{delta-factor-Forch}.
\end{proof}

%\akif{Here are two pacifying  items}
\begin{remark}  

\begin{enumerate}

\item As $\beta\to 0$ factor $\delta\to 1$, and we consequently get Peaceman well-block radius as in the Theorem  \ref{Peaceman-basic}.
\item Evidently non-linear functional equation \eqref{delta-factor-Forch} has a solution for any $q>0.$

\end{enumerate}

\end{remark}

%----------------------------------------------------------------------------------------------------

%-------------------------------------------------------------
\section{Impact of the Correction in the Dake Formula}
\subsection{Dake Formula and D-factor}\label{Dake}

To compute fluid inflow to well in the case of Forchheimer non-linear flow, reservoir simulators use the so-called D-factor, or dynamic skin-factor. It is assumed that Peaceman radius is still applicable. 

However, a nonlinear (rate-dependent) correction is introduced in the denominator of the inflow formula. The derivation can be found in \cite{Dake-book} and is briefly given below. As before, we neglect compressibility of the fluid for simplicity. However, the derivation can be directly extended to compressible case by introduction of pseudo-pressure \cite{Dake-book}.

Assume flow is radial and $R_e$ is radius of a contour with specified pressure $p_e$. Equation \eqref{Radial-Two-Term} is integrated from $R_w$ to $R_e$ to relate pressure drop $p_e-p_w$ to flow rate $q=2 \pi r h v_r$
\begin{equation}\label{Dake_two_term}
   p_e-p_w=\frac{\alpha}{2\pi}q\ln\frac{R_e}{R_w}+\beta \frac{q^2}{4\pi^2}\left(\frac{1}{R_w}-\frac{1}{R_e}\right)   
\end{equation}
Now if $R_e$ is large enough to neglect $\frac{1}{R_e}$, the dimensionless pressure drop becomes
\begin{equation}\label{Dake_dimensionless}
   \frac{{2\pi}(p_e-p_w)}{\alpha q}=\ln\frac{R_e}{R_w}+\frac{\beta}{\alpha} \frac{q}{2\pi}\frac{1}{R_w} 
\end{equation}

Here the first term in RHS is equivalent to the Darcy case, and the second term is a function (linear) of $q$, with 
\begin{equation}\label{D-factor}
D=\frac{\beta}{\alpha}\frac{1}{2{\pi}R_w}
\end{equation}
being the D-factor.

In reservoir simulators this result is used for Forchheimer flow as follows. As in the linear case, inflow to the well is assumed radial within the grid block containing the well. Formula \eqref{Dake_dimensionless} with $R_e=R_0$ and $p_e=p_0$ is used to relate well flow rate $q$ to the difference of grid block pressure $p_0$ and well pressure $p_w$. The key problem is that $R_0$ is computed by \eqref{R-0-delta}, just as in the linear case:
\begin{equation}\label{Dake_simulators}
   p_0-p_w=\frac{\alpha}{2\pi}q\left(\ln\frac{R_{0_{Darcy}}}{R_w}+Dq\right)   
\end{equation}
In fact, the use of $R_{0_{Darcy}}$ in \eqref{Dake_simulators} is incorrect, since the Peaceman formula was derived for Darcy flow only. The correct formula for Forchheimer flow can be obtained by combining \eqref{p_0-p_w-F}, \eqref{R_0-for-Forch} and \eqref{delta-factor-Forch}. \newline Recall \eqref{p_1}-\eqref{q} and substitute \eqref{delta-factor-Forch} multiplied by $-\frac{\pi}{2}$:

\begin{equation}\label{Dake_correct_derivation1}
  p_0-p_w=\frac{\alpha}{2\pi}q\left(\ln\frac{\Delta}{R_w}-\frac{\pi}{2}-\beta\frac{q}{2\alpha\pi}\left(\frac{1}{\Delta\cdot e^{-\delta\frac{\pi}{2}}}-\frac{1}{\Delta}\right)\right)
  +\beta\frac{q^2}{4\pi^2}\left(\frac{1}{R_w}-\frac{1}{\Delta\cdot e^{-\delta\frac{\pi}{2}}}\right),
\end{equation}

or rearranging:

\begin{equation}\label{Dake_correct_derivation2}
  p_0-p_w=\frac{\alpha}{2\pi}q\ln\frac{\Delta\cdot e^{-\frac{\pi}{2}}}{R_w}+\beta\frac{q^2}{4\pi^2}\left(\frac{1}{R_w}-\frac{1}{\Delta\cdot e^{-\delta\frac{\pi}{2}}}-\frac{1}{\Delta}+\frac{1}{\Delta\cdot e^{-\delta\frac{\pi}{2}}}\right),
\end{equation}

which finally gives:

\begin{equation}\label{Dake_correct}
    p_0-p_w=\frac{\alpha}{2\pi}q\ln\frac{R_{0_{Darcy}}}{R_w}+\beta \frac{q^2}{4\pi^2}\left(\frac{1}{R_w}-\frac{1}{\Delta}\right).
\end{equation}

This result is very similar to \eqref{Dake_simulators} except for the last term with $\frac{1}{\Delta}$. Thus it turns out that for coarse grids with large enough $\Delta$ the correct formula \eqref{Dake_correct} is almost equivalent to the one used by reservoir simulators. But as $\Delta$ goes smaller, the correction gets significant, and the D-factor becomes grid-dependent (dependent on $\Delta$). And note that even for quite large $\Delta$ the correction can become significant with increasing $q$.

\section{Conclusive Remarks and Discussion}\label{conclusion}
We revisited Peaceman well-block radius formula and proved its analog for Green function in the annular domain. Main algorithm is based on 'sewing' of material balance equation \eqref{Mat-Bal-0} with the analytical equation of flow. For the linear case, we introduced the concept of Peaceman well-posedness and rigorously obtained Peaceman classical well-block radius for Darcy flow near well and linear material balance equation on coarse numerical grid. We also considered the following locally nonlinear problem: we kept material balance equation to be linear but allowed non-linearity in the flow equation \eqref{dup-0}. In this case, for linear balance equation on the coarse grid and Forchheimer two terms law near the well, we proved that well-block radius  $R_0$ should depend on $q$.  We implemented this algorithm in Section \ref{Two-Term} and Section \ref{Dake}, and obtained formulae for $R_0$ which is very close to traditional one used in flow simulators for big $\Delta$, but quite different for small $\Delta$ or large $q$. This approach is planned to be extended  for different classes of non-Darcy flows. Wide class of Non-Darcy flows was presented in the articles \cite{ibragim-prod-ind-gas, ibragim-prod-ind-PSS}.

\bibliographystyle{plain}

%%%%%%%%%%%%%%%%%%%%%%%%%%%%
%%%%%%%%%%%%%%%%%%%%%%%%%%%%%

\end{document}